\documentclass[12pt]{article}

\usepackage{amsmath,amssymb}
\usepackage{amsfonts}
\usepackage{multicol} %tabla
\usepackage{array} %tabla
\usepackage{lscape} %p\'aginas en horizontal
\usepackage{xcolor} %colores
\usepackage{enumerate}
\usepackage[colorlinks=true,linkcolor=blue,urlcolor=black,bookmarksopen=true]{hyperref}
\usepackage{bookmark}
\usepackage{authblk}

\input xy
\xyoption{all}

\newtheorem{Def}{Definition}[section]
\newtheorem{Prop}[Def]{\bf Proposition}

\newtheorem{Lemma}[Def]{\bf Lemma}

\newtheorem{Teo}[Def]{\bf Theorem}
\newtheorem{Obs}[Def]{\bf Observation}
\newtheorem{Example}[Def]{\bf Example}

\newcommand{\rr}{\Rightarrow} %implicación de Heyting
\newcommand{\sN}{\mbox{$\cal SN$}} %variedad de semi-Nelson
\newcommand{\sH}{\mbox{$\cal SH$}} %variedad de semi-Heyting
\newcommand{\He}{\mbox{$\cal H$}} %variedad de Heyting
\newcommand{\Ne}{\mbox{$\cal N$}} %variedad de Nelson
\newcommand{\dsh}{\mbox{$\cal DSH$}} %variedad de semi-Heyting dually
\newcommand{\dsn}{\mbox{$\cal DSN$}} %variedad de semi-Nelson duak
\newcommand{\ton}{\mbox{$\ \to_\text {\tiny N}\ $}} %implicación de Nelson en semi-Nelson
\newcommand{\toh}{\mbox{$\ \Rightarrow_\text {\tiny H}\ $}} %implicación de Heyting en semi-Heyting

\newcommand{\class}[1]{[\![#1]\!]} %clase de equivalencia
\newcommand{\Ds}{\mathcal{D}_s}
\newcommand{\tic}{^\dagger} %operación binaria de dually semiHeyting
\newcommand{\ba}{\mathbf{A}}

\newenvironment{proof}{\noindent\bf Proof. \rm}{\hfill$ \blacksquare$}

\numberwithin{equation}{section}

\title{A Categorial Equivalence for semi-Nelson algebras}
\author[1]{Juan Manuel Cornejo}
\author[1]{Andr\'es Gallardo}
\author[1,*]{Ignacio Viglizzo}
\affil[1]{INMABB UNS CONICET, Departamento de Matem\'atica - Av. Alem 1253, Bahía Blanca, Buenos Aires, Argentina.}
\affil[*]{Corresponding author: {\tt viglizzo@gmail.com}, ORCiD:0000-0002-5303-623X }

\begin{document}

\maketitle

\begin{abstract}
We present a category equivalent to that of semi-Nelson algebras. The objects in this category are pairs consisting of a semi-Heyting algebra and one of its filters. The filters must contain all the dense elements of the semi-Heyting algebra and satisfy an additional technical condition. We also show that in the case of dually hemimorphic semi-Nelson algebras, the filters are not necessary and the category is equivalent to that of dually hemimorphic semi-Heyting algebras.
\end{abstract}

\section{Introduction}

Semi-Heyting algebras were introduced in 1985 by H.~Sankappanavar \cite{sankappanavar1985semi} as a variety generalizing that of Heyting algebras while keeping many of its good features like being distributive pseudocomplemented lattices and having their congruences determined by filters. 

On the other hand, D.~Vakarelov provided in \cite{vakarelov77Nlattices} a way of constructing Nelson algebras from Heyting ones, by means of what is now known as twist product, thus extending work by Kalman \cite{kalman58lattices}. This program was continued by Sendlewski \cite{sendlewski1984investigations}, \cite{sendlewski90nelson}, who gave a full representation of Nelson algebras using Heyting algebras and a boolean congruence on them. Later, L.~Monteiro and I.~Viglizzo \cite{viglizzo99algebras} and \cite{monteiro19construction}, used filters instead of congruences to study this representation.

It turned out that another one of the nice features of semi-Heyting algebras is that Vakarelov's construction can be carried out on them as well. The resulting algebras form the variety of semi-Nelson algebras \cite{cornejo16semiNelson}. Vakarelov's construction gives a categorial equivalence between semi-Heyting algebras and \textit{centered} semi-Nelson algebras \cite{cornejo18categorical}, but this leaves out many semi-Nelson algebras. In this work we prove that considering pairs of semi-Heyting algebras and one of its filters satisfying some extra conditions, we can construct all the semi-Nelson algebras. This gives a good representation that can be seen as a categorial equivalence, and permits understanding semi-Nelson algebras in terms of the better-known semi-Heyting algebras.

H.~Sankappanavar defined in \cite{sankappanavar2011expansions} the variety of semi-Heyting algebras with a dual hemimorphism, or dually hemimorphic semi-Heyting algebras,  $\dsh$, as an expansion of semi-Heyting algebras with an unary operator that is a common generalization of both De Morgan and pseudocomplemented algebras. J.~M.~Cornejo and H.~San Mart\'in applied Vakarelov's construction to this expansion to obtain the variety of dually hemimorphic semi-Nelson algebras, $\dsn$\cite{cornejo18categorical}. We prove here that dually hemimorphic semi-Nelson algebras are centered, and therefore they can be represented by dually hemimorphic semi-Heyting algebras.

The paper is organized as follows: Section 2 introduces the varieties of algebras we will be dealing with, recalls some of the basic results on them that we will be using, and provides background on Vakarelov's construction. Section 3 details the representation of semi-Nelson algebras by pairs consisting of a semi-Heyting algebra and one of its filters containing the dense elements of the algebra and satisfying an additional condition, proving this representation is an equivalence of two categories. The final section deals with the case of dually hemimorphic semi-Nelson algebras.

\section{Preliminaries}

\subsection{Definitions and elementary properties}\label{02022020}

In this section we will recall the definitions and some basic properties of the varieties of Heyting and semi-Heyting algebras,  Nelson and semi-Nelson algebras, pseudocomplemented lattices, and dually hemimorphic semi-Heyting and semi-Nelson algebras. 

\begin{Def}\label{defsemihey}
	
	A \emph{semi-Heyting algebra} is an algebra $\mathbf{H}=\langle H; \land,\lor,\rr,0,1 \rangle$ of type $(2,2,2,0,0)$ such that the following conditions are satisfied for all $x,y,z\in A$:
	
	\begin{enumerate}[(SH1)]
		\item $\langle H,\land,\lor,0,1 \rangle$ is a bounded lattice,
		\item $x\land(x\rr y)=x\land y$, \label{infimoeimplicaSH}
		\item $x\land(y\rr z)=x\land((x\land y)\rr(x\land z))$, \label{meterinfimoadentrodeimplicaSH}
		\item $x\rr x=1$. \label{autoimplicaSH}
	\end{enumerate}
	
	A semi-Heyting algebra is a \emph{Heyting algebra} if it satisfies the identity:
	\begin{enumerate}
	    \item [(H)] $(x\land y)\rr x=1$.
	\end{enumerate}
\end{Def}

On a semi-Heyting algebra $\mathbf{H}$ one can always define the term $x\toh y=x\rr(x\land y)$. With this operation, the system $\langle H,\land,\lor,\toh,0,1 \rangle$ is a Heyting algebra \cite{Abad2013semi}.

Nelson algebras are the algebraic counterpart of D. Nelson's constructive logic with strong negation. The equations presented here were proved to be a complete and independent axiomatization in \cite{monteiro96axiomes}:

\begin{Def}\label{defnel}
	
	A \emph{Nelson algebra} is an algebra $\mathbf{A}=\langle A; \land,\lor,\to,\sim,1 \rangle$ of type $(2,2,2,1,0)$ such that the following conditions are satisfied for all $x,y,z\in A$:
	
	\begin{enumerate}[(N1)]
		\item $x\land(x\lor y)=x$, \label{primercondicretN}
		\item $x\land(y\lor z)=(z\land x)\lor(y\land x)$, \label{segundacondicretN}
		\item $\sim\sim x=x$, \label{involucionN}
		\item $\sim(x\land y)=\sim x\lor \sim y$, \label{deMenInfN}
		\item $x\land \sim x=(x\land \sim x)\land (y\lor \sim y)$, \label{kleeneN}
		\item $x\to x=1$, \label{ximpxigual1N}
		\item $x\to(y\to z)=(x\land y)\to z$, \label{relacImplicaInf2N}
		\item $x\land (x\to y)=x\land(\sim x\lor y)$. \label{relacImplicaInfN}
	\end{enumerate}
\end{Def}

The variety of semi-Nelson algebras  was introduced in \cite{cornejo16semiNelson} as a generalization of Nelson algebras.

\begin{Def} \label{defseminel}
A \emph{semi-Nelson algebra} is an algebra $\mathbf{A}=\langle A; \land, \lor,\to, \sim, 1 \rangle$ of type $(2,2,2,1,0)$ such that the following conditions are satisfied for all $x,y,z\in A$:

    \begin{enumerate}[(SN1)]
    	\item $x\land(x\lor y)=x$, \label{identidad_absorc}
    	\item $x\land(y\lor z)=(z\land x) \lor(y\land x)$, \label{identidad_distrib_inf}
    	\item $\sim\sim x=x$, \label{identidad_doble_neg}
    	\item $\sim(x\land y)=\sim x\lor\sim y$, \label{identidad_ditrib_neg}
    	\item $x\land\sim x=(x\land\sim x)\lor(y\lor\sim y)$, \label{identidad_kleene}
    	\item $x\land(x\ton y)=x\land(\sim x\lor y)$, \label{identidad_implic_negacion}
    	\item $x\ton(y\ton z)=(x\land y)\ton z$, \label{identidad_implica_infimo}
    	\item $(x\ton y)\ton[(y\ton x)\ton[(x\to z)\ton(y\to z)]]=1$,% \label{identidad_congr_derecha}
    	\item $(x\ton y)\ton[(y\to_N x)\ton[(z\to x)\ton(z\to y)]]=1$,% \label{identidad_congr_izquierda}
    	\item $(\sim(x\to y))\ton(x\land\sim y)  =  1$,\label{identidad_semi_neg_implica_impl}
    	\item $(x\land\sim y)\ton(\sim(x\to y))  =  1$,\label{identidad_semi_impl_implica_neg}
    \end{enumerate}
where $x \ton y := x \to (x \land y)$.
\end{Def}

For a semi-Nelson algebra $\mathbf{A}$, the system $\langle A,\land,\lor,\ton,\sim,1 \rangle$ is a Nelson algebra \cite{cornejo16semiNelson}.

We will denote by $\He$, $\sH$, $\Ne$ and $\sN$ the varieties of Heyting, semi-Heyting, Nelson and semi-Nelson algebras respectively.

Heyting algebras, as is well known, are pseudocomplemented (with $x^*=x\to 0$ as the pseudocomplement of $x\in H$) in the sense of the following definition:

\begin{Def}\label{defpseudocomplementedlattices}
An algebra $\mathbf{A}=\langle A; \land,\lor,^*,0,1\rangle$ is a \emph{pseudocomplemented lattice} if the following conditions hold:

\begin{enumerate}[PS1)]
    \item $\langle A; \land,\lor,0,1\rangle$ is a bounded lattice,
    \item $x\land(x\land y)^*=x\land y^*$,
    \item $0^*=1$ and $1^*=0$.
\end{enumerate}
\end{Def}

If we define in a semi-Heyting algebra the term $x^c=x\rr0$, it turns out that $x^c$ is the pseudocomplement of $x$ in the sense of the previous definition. Moreover, $x\toh 0=x\Rightarrow (x\land 0)=x\Rightarrow 0$,
so $x^*$ and $x^c$ coincide in $\sH$. From now on we will use indistinctly the notation $x^*$ for both $x^c$ and $x^*$.

We will use some properties of pseudocomplemented lattices. Their proof can be found in \cite{balbes1974distributive}:

\begin{Prop}
If $\mathbf A=\langle A; \land,\lor,^*,0,1\rangle$ is a pseudocomplemented lattice, then for every $x, y \in A$, if $x\land y=0$, then $x\leq y^*$. It also holds that  $(x\lor y)^*=x^*\land y^*$.
\end{Prop}

\begin{Def} \label{dense_element}
Let $\mathbf{A}=\langle A; \land,\lor,^*,0,1\rangle$ be a pseudocomplemented lattice. We say that an element $a\in A$ is \emph{dense} if $a^*=0$. We denote by $\Ds(A)$ the set of all dense elements of $A$.
\end{Def}

The following characterization of dense elements is going to be useful:

\begin{Lemma}\label{caractdensos}
If $\mathbf H\in\sH$, then $x\in\Ds(H)$ if and only if $x=y\lor y^*$ for some $y\in H$.
\end{Lemma}

For the results in section \ref{duallyhemimorphicsection}, we now define the varieties of dually hemimorphic semi-Heyting and semi-Nelson algebras. The former were introduced by H.P. Sankappanavar in \cite{sankappanavar2011expansions}, and the latter were presented in \cite{cornejo19dually}.

\begin{Def}\label{defdualsemiheyting}
An algebra $\mathbf{A}=\langle A; \land,\lor,\to,\tic,0,1 \rangle$ of type $(2,2,2,1,0,0)$ is said to be a \emph{dually hemimorphic semi-Heyting algebra} if $\langle A; \land,\lor,\to,0,1 \rangle$ is a semi-Heyting algebra and the following equations are satisfied:

\begin{enumerate}[DSM1)]
    \item $0\tic=1$,\label{op_tilde_0}
    \item $1\tic=0$,\label{op_tilde_1}
    \item $(x\land y)\tic=x\tic\lor y\tic$.\label{distib_tilde_inf}
\end{enumerate}
\end{Def}

We write $\dsh$ to denote the variety of dually hemimorphic semi-Heyting algebras.

\begin{Def}\label{defdualseminelson}
An algebra $\mathbf{A}=\langle A; \land,\lor,\to,\sim,',1 \rangle$ of type $(2,2,2,1,0)$ is said to be a \emph{dually hemimorphic semi-Nelson algebra} if $\langle A; \land,\lor,\to,\sim,1 \rangle$ is a semi-Nelson algebra and the following equations are satisfied:

\begin{enumerate}[DSN1)]
    \item $(\sim 1)'=1$,\label{dhsn_1}
    \item $1'\to(\sim1)=1$,\label{dhsn_2}
    \item $((x\to y)\land(y\to x)\land x')\to((x\to y)\land(y\to x)\land y)'=1$,\label{dhsn_3}
    \item $\sim x'\to(\sim x\land(x'\to x))=1$,\label{dhsn_4}
    \item $(\sim x\land(x'\to x))\to\sim x'=1$,\label{dhsn_5}
    \item $(x\land y)'\to(x'\lor y')=1$,\label{dhsn_6}
    \item $(x'\lor y')\to(x\land y)'=1$.\label{dhsn_7}
\end{enumerate}

We use the convention that the unary operation $'$ has higher priority than $\sim$, so the expression $\sim x'$ means $\sim(x')$.
\end{Def}

We write $\dsn$ to denote the variety of dually hemimorphic semi-Nelson algebras.

\subsection{Vakarelov's Construction} \label{vakarelovconstruction}\label{quotients}

In this section we review some well known results that established the connection between Heyting and Nelson algebras, and later allowed the definition of (dually hemimorphic) semi-Nelson algebras as a variety constructed from the one of (dually hemimorphic) semi-Heyting algebras. 

Let $\mathbf{A}\in\He$. We denote with $V_k(A)$ the set $\{ (a,b)\in A^2 : a\land b=0 \}$ and with $\mathbf{V_k(A)}$ the system $\langle V_k(A); \sqcap,\sqcup,\to,\sim,\top \rangle$ algebrized in the following way

    \begin{enumerate}[(V1)]
        \item $(a,b)\sqcap(c,d)=(a\land c,b\lor d)$,
        \item $(a,b)\sqcup(c,d)=(a\lor c,b\land d)$,
        \item $(a,b)\to(c,d)=(a\rr c,a\land d)$,
        \item $\sim (a,b)=(b,a)$,
        \item $\top=(1,0)$.
    \end{enumerate}

\begin{Obs}
    From the previous definitions we can deduce the rule $$(\mbox{\rm V}6)\ (a,b)\ton(c,d)=(a\toh c,a\land d).$$
\end{Obs}

It was shown in \cite{vakarelov77Nlattices} that $\mathbf{V_k(A)}\in\sN$.

Vakarelov's construction can be generalized the following way: if $\mathbf{A}\in\He$, and $F$ is a filter of $\mathbf{A}$, we define the structure $$N(A,F):=\{(a,b)\in A^2 : a\land b=0, a\lor b\in F \}.$$ It was proved in \cite{viglizzo99algebras,monteiro19construction} that $N(A,F)$, algebrized with (V1)-(V5), is well defined and it is a Nelson algebra. This algebra will be denoted by $\mathbf{N(A,F)}$.

A different generalization was to start from a semi-Heyting algebra, which lead to the definition of semi-Nelson algebras \cite{cornejo16semiNelson}: if $\mathbf{A}$ is a semi-Heyting algebra, then $\mathbf{V_k(A)}$ is a semi-Nelson algebra.

Going in the other direction, we can go from semi-Nelson algebras to semi-Heyting algebras by means of the following quotient:

Let $\mathbf{A}= \langle A, \land,\lor, \to, \sim, 1\rangle\in\sN$, and define:
$$x\equiv y\mbox{ if and only if } x\to y = 1\mbox{ and } y\to x = 1.$$

It can be proved that ``$\equiv$'' is an equivalence relation and a congruence with respect of the operations $\land$, $\lor$ and $\to$. We denote with $\class{a}$ the equivalence class of an element $a\in A$ under the relation $\equiv$. 

We consider $\mathbf{sH}(\mathbf{A}) = \langle A/_\equiv, \cap,\cup, \Rightarrow, \bot, \top\rangle$, where
\begin{multicols}{2}
	\begin{itemize}
		\item $\bot = \class{\sim 1}$,
		\item $\top = \class{1}$,
		\item $\class{x} \cap \class{y} = \class{x \land y}$,
		\item $\class{x} \cup \class{y} = \class{x \lor y}$,
		\item $\class{x} \rr \class{y} = \class{x \to y}$.
	\end{itemize}
\end{multicols}

Then $\mathbf{sH}(\mathbf{A})\in\sH$ (\cite{cornejo16semiNelson}).

\vspace{5mm}

We will use the next result in the following sections:

\begin{Prop}\label{propertiesseminelson} {\rm \cite[Lemma 2.7 (h) and (l)]{cornejo16semiNelson}}
Let $\mathbf{A}\in\sN$, and $a,b,c\in A$. Then the following properties hold:
\begin{enumerate}[(1)]
    \item $a\ton b=b\ton a=1$ if and only if $a\to b=b\to a=1$,\label{equivalenciadeclasesnelson}
    \item If $a\ton b=1=b\ton c=1$, then $a\ton c=1$.\label{transitivitynelson}
\end{enumerate}
\end{Prop}

Part (1) of the previous proposition  tells us that the congruence $\equiv$ is the same as the one that can be defined if we replace $\to$ with $\ton$.

The following representation theorem  establishes the connection between semi-Heyting and semi-Nelson algebras using Vakarelov's construction and the quotient by the relation $\equiv$. 

\begin{Teo} \label{isosemiHeytingsemiNelson} {\rm (\cite[Theorem 5.3 and Corollary 5.2]{cornejo16semiNelson})}
If $\mathbf{A}\in\sH$, then $\mathbf{A}$ is isomorphic to $\mathbf{sH}(\mathbf{V_k}(\mathbf A))$ through the isomorphism $g(a)=\class{(a,a^*)}$. Also, if $\mathbf{A}\in\sN$, then $\mathbf{A}$ is isomorphic to a subalgebra of $\mathbf{V_k(sH(\mathbf A))}$ through the application $h(a)=(\class{a},\class{\sim a})$.
\end{Teo}

\section{Representations}\label{representations}

Theorem \ref{isosemiHeytingsemiNelson} gives a representation of any semi-Nelson algebra $\mathbf{A}$ as a subalgebra of $\mathbf{V_k}(\mathbf{sH}(\mathbf{A}))$, but there can be more than one non-isomorphic semi-Nelson algebra having the same quotient semi-Heyting algebra $\mathbf{sH}(\mathbf{A})$ (up to isomorphism).  In order to get a sharper representation we are going to consider pairs consisting of a semi-Heyting algebra and one of its filters satisfying extra conditions. To be more precise, we obtain a categorial equivalence between the category of semi-Nelson algebras and a category whose objects are pairs $(\mathbf{H}, F)$ where $\mathbf{H}$ is a semi-Heyting algebra and $F$ is a filter of $\mathbf{H}$ that contains all the dense elements of $\mathbf{H}$ and satisfies an extra condition.

Recall the construction $N(A,F)=\{(a,b)\in A^2 : a\land b=0, a\lor b\in F \}$ for some Heyting algebra $\mathbf{A}$, and $F$ a filter of $\mathbf{A}$. For any $\mathbf{A}\in\He$, we consider $\mathbf{V_k(A)}$, and a subalgebra $\mathbf{S}$ such that the projection over the first component, $\pi_1$, verifies $\pi_1(S)=A$. It was shown in \cite{viglizzo99algebras} that there is a filter $F\subseteq A$ containing the dense elements of $A$ such that $\mathbf S=\mathbf{N(A,F)}$.

To extend these results to the variety $\sH$, we define:

\begin{Def}\label{defifiltro}
Let $\mathbf{H}\in\sH$. A subset $F\subseteq H$ is an $i$-filter, if the following conditions hold:
\begin{enumerate} [\text{I}F1)]
    \item $F$ is a lattice filter of $\mathbf H$.
    \item $\Ds(H)\subseteq F$.
    %$x\lor x^*\in F$.
    \item If $z\lor t\in F$, then for all $x\in H$, $(x\rr z)\lor(x\land t)\in F$.
\end{enumerate}
\end{Def}

It is immediate from the definition that every $i$-filter is a filter. The converse is not true, as  the following example shows:

\begin{Example}\label{ejemplofiltronoifiltro}
	Consider the semi-Heyting algebra $\mathbf{A}=\langle A; \land,\lor,$ $\rr,1 \rangle$, with the operation $\rr$ indicated in the table:
	\vspace{2mm}
	
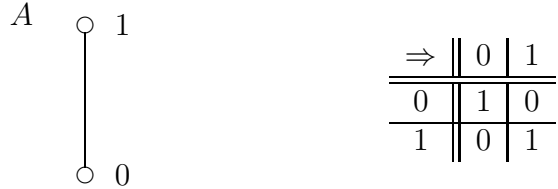
	\begin{minipage}{7cm}
		\begin{center}
			\setlength\unitlength{1mm}
			% \begin{picture}(90,60)(-10,-10)
			\begin{picture}(70,40)(10,-10)
			% puntos
			\put(45,20){\circle{2}}
 			\put(45,0){\circle{2}}
			
			% rotulos a los puntos
			\put(50,0){\makebox(0,0){$0$}}
			\put(50,20){\makebox(0,0){$1$}}

			%lineas|
			\put(45,1){\line(0,1){18}}
			%rotulo
			\put(35,20){$A$}
			\end{picture}
		\end{center}
	\end{minipage}
	\begin{minipage}{3cm}
		\begin{center}
			\begin{tabular}{c||c|c}
				$\Rightarrow$&$0$&$1$\\\hline\hline
				$0$&$1$&$0$\\ \hline
				$1$&$0$&$1$
			\end{tabular}
		\end{center} 
	\end{minipage}
	
The subset $F=\{1\}$ of $\mathbf A$ is a filter, and it satisfies $IF2)$ (since the only dense element is 1). But $F$ does not satisfy $IF3$) because $1\lor 0=1\in F$, but $(0\rr 1)\lor(0\land0)=0\lor0=0\notin F$. %Hence, $F$ is not an $i$-filter.
\end{Example}

\begin{Obs}\label{buenadefdeimplicavakarelov}
If we consider the structure $\mathbf{ N(H,F)}$, with $\mathbf{H}\in\sH$, an $i$-filter $F$, 
and $(a,b),(c,d)\in N(H,F)$, then (1) $c\land d=0$ and (2) $c\lor d\in F$. Notice that
$$(a\rr c)\land(a\land d)=a\land(a\rr c)\land d\underset{(SH\ref{infimoeimplicaSH})}{=}(a\land c)\land d=a\land(c\land d)\underset{(1)}{=}a\land 0=0.$$

Also, using (2) and $IF3)$ we have that $(a\rr c)\lor(a\land d)\in F$, so the pair $(a\rr c,a\land d)$ is in $N(H,F)$ and we can use it to define $(a,b)\to(c,d)=(a\rr c,a\land d)$.
\end{Obs}

\begin{Lemma}\label{NHFessiminelsonalgebra}
If $\mathbf{H}\in\sH$ and $F$ is an $i$-filter of $\mathbf H$, then the system \linebreak $\langle N(H,F);$ $\sqcap,\sqcup,\to,\sim,\top \rangle$  with the operations defined as in  (V1)-(V5) is a semi-Nelson algebra.
\end{Lemma}

\begin{proof}
Using that $F$ is a filter we can prove that the operations $\sqcap$, $\sqcup$, $\top$ and $\sim$ are well defined in $N(H,F)$ (see \cite{monteiro19construction}). Also, by observation \ref{buenadefdeimplicavakarelov}, it follows that the operation $\to$ is well defined. Furthermore, following the proof of Theorem 4.1 in\cite{cornejo16semiNelson},  $\mathbf{N(H,F)}$ verifies the axioms $SN1)$-$SN11)$ so it is is a semi-Nelson algebra.
\end{proof}

\begin{Lemma} \label{16012020}
Let $\mathbf{H}\in\sH$. If $\mathbf S$ is a subalgebra of $\mathbf{V_k}(\mathbf{H})$, such that $\pi_1(S)=H$ (with $\pi_1$ the projection over the first component of ${V_k}(H)$), then there exists an $i$-filter $E$ of $\mathbf H$ such that $S=N(H,E)$.
\end{Lemma}

\begin{proof}
Let $E=\{x\in H : x=a\lor b \text{ for some } (a,b)\in S\}$. We are going to prove that $E$ is an $i$-filter.

It was shown in \cite{viglizzo99algebras} that if $\mathbf H\in\He$, then $E$ is a filter of $\mathbf H$. But that proof can be extended naturally to the variety $\sH$, because it only uses lattice properties.

Consider a dense element $x\in H$. Since $\pi_1(S)=H$, then there exists $y\in H$ such that $(x,y)\in S$. Besides, since $S$ is a subalgebra of $\mathbf{V_k(H)}$, $(0,1)\in S$, and therefore $(x,y)\to(0,1)=(x\rr0,x\land1)=(x^*,x)=(0,x)\in S$. Hence, $x\in E$ because $x=x\lor0$, and we conclude that $E$ satisfies $IF2$).

Suppose now that for $z,t\in H$ we have $z\lor t\in E$,

Since $z\lor t\in E$, there exists a pair $(a,b)\in S$ such that $z\lor t=a\lor b$. But $S$ is a subalgebra of $\mathbf{V_k(H)}$, so $\sim(a,b)=(b,a)\in S$, and therefore $(a,b)\cup(b,a)=(a\lor b,a\land b)=(a\lor b,0)\in S$. This means that $(z\lor t,0)\in S$, and hence
\begin{equation}\label{19022020_1}
    \sim(z\lor t,0)=(0,z\lor t)\in S.
\end{equation}

Using the hypothesis $\pi_1(S)=H$, and $x,z,t\in H$, it follows that there exist some $x',z',t'\in H$ such that $(x,x'), (z,z'), (t,t')\in S$. This implies that the next term belongs to $S$:
$$((x,x')\to(z,z'))\cup((x,x')\cap(t,t'))=(x\Rightarrow z,x\land z')\cup(x\land t,x'\lor t')=$$ 
\begin{equation}\label{19022020_2}
    ((x\Rightarrow z)\lor(x\land t),(x\land z')\land(x'\lor t')).
\end{equation}

Since $(x\land z')\land(x'\lor t')=(x\land z'\land x')\lor(x\land z'\land t')=0\lor(x\land z'\land t')=x\land z'\land t'$, (\ref{19022020_2}) yields
\begin{equation}\label{19022020_3}
    ((x\Rightarrow z)\lor(x\land t),x\land z'\land t')\in S.
\end{equation}

Combining (\ref{19022020_1}) with (\ref{19022020_3}) we get
\begin{equation}\label{19022020_4}
    ((x\Rightarrow z)\lor(x\land t),x\land z'\land t')\cup(0,z\lor t)=((x\Rightarrow z)\lor(x\land t),x\land z'\land t'\land(z\lor t))\in S.
\end{equation}

Since $x\land z'\land t'\land(z\lor t)=(x\land z'\land t'\land z)\land(x\land z'\land t'\land t)=(x\land t'\land(z\land z'))\lor(x\land z\land(t\land t'))=0\lor0=0$, from (\ref{19022020_4}) it follows that
$$((x\Rightarrow z)\lor(x\land t),0)\in S.$$ 

By the definition of $E$ it follows that $$(x\Rightarrow z)\lor(x\land t)\lor0=(x\Rightarrow z)\lor(x\land t)\in E.$$

Hence, $E$ is an $i$-filter.

Let us prove now that $N(H,E)=S$. It is clear that $S\subseteq N(H,E)$. If we have $(a,b)\in N(H,E)$, then $a\land b=0$, and $a\lor b\in E$, so there exists $(c,d)\in S$ such that $c\lor d=a\lor b$. Then, as before $(c,d)\cup(d,c)=(c\lor d,0)=(a\lor b,0)\in S$. On the other hand, since $\pi_1(S)=H$, there exists some $b'\in H$ such that $(b,b')\in S$, and therefore $(b,b')\to(0,1)=(b^*,b)\in S$. 
By properties of the pseudocomplement, $a\land b=0$ implies that $a\leq b^*$ and hence $(a\lor b)\land b^*=(a\land b^*)\lor(b\land b^*)=a\lor 0=a$, so we can conclude that $(a\lor b,0)\cap(b^*,b)=((a\lor b)\land b^*,0\lor b)=(a,b)\in S$.
\end{proof}

\begin{Teo}
Every semi-Nelson algebra $\mathbf{A}$ is isomorphic to an algebra of the form $\mathbf{N(H,F)}$, for some semi-Heyting algebra $\mathbf{H}$ and some $i$-filter $F$ of $\mathbf{H}$.
\end{Teo}
\begin{proof}
By Theorem \ref{isosemiHeytingsemiNelson}, every semi-Nelson algebra $\mathbf{A}$ can be obtained as a subalgebra of $\mathbf{V_k}(\mathbf{sH}(A))$, and clearly $\pi_1(\mathbf{V_k}(\mathbf{sH}(A)))=\mathbf{sH}(A)$, so $\mathbf{A}$ is isomorphic to $\mathbf N(\mathbf{sH}(\mathbf A),\mathbf E)$, where $E$ is the $i$-filter of Lemma \ref{16012020}.
\end{proof}

\medskip

Much in the same manner as for the case of Heyting algebras, for $\mathbf{H}\in\sH$, the $i$-filters $E$ such that $\pi_1(N(H,E))=H$ are exactly the ones that contain the dense elements of $H$:

\begin{Lemma} \label{ifiltersthatcontaindenseelem}
If $\mathbf{H}\in\sH$, and $F$ is an $i$-filter of $\mathbf H$, then $\pi_1(N(H,F))=H$.
\end{Lemma}

\begin{proof}
If $x\in H$, by Lemma \ref{caractdensos}, we have that $x\lor x^*\in\Ds(H)\subseteq F$. Also, $x\land x^*=0$ is valid, so $(x,x^*)\in N(H,F)$ and $x\in\pi_1(N(H,F))$.
\end{proof}

\subsection{Categorial equivalence}

We define now the category $\mathbf{sHF}$, which has as objects pairs $(\mathbf{H},F)$ (where $\mathbf{H}\in\sH$, and $F$ is an $i$-filter of $H$),  and as morphisms, functions $f:(\mathbf H,F)\longrightarrow (\mathbf H',F')$ such that $f:H\longrightarrow H'$ is a homomorphism of semi-Heyting algebras and $f(F)\subseteq F'$.

Consider now the category $\mathbf{sN}$ of semi-Nelson algebras and their morphisms; our objective now is to establish an equivalence between these categories. For this we define functors $\alpha:\mathbf{sHF}\longrightarrow\mathbf{sN}$ and $\beta:\mathbf{sN}\longrightarrow\mathbf{sHF}$.

\begin{Prop}\label{functorfromSHFtoSN}
The application $\alpha:\mathbf{sHF}\longrightarrow\mathbf{sN}$ defined on the objects of $\mathbf{sHF}$ by $\alpha((\mathbf H,F))=\mathbf{N(H,F)}$, and on morphisms by $\alpha(f)(a,b)=(f(a),f(b))$ for all $(a,b)\in N(H,F)$, is a functor from $\mathbf{sHF}$ to $\mathbf{sN}$.
\end{Prop}

\begin{proof}
Let us check that $\alpha$ is well defined. If $(\mathbf H,F)$ is an object of $\mathbf{sHF}$, then $\alpha((\mathbf H,F))=\mathbf{N(H,F)}$ is a semi-Nelson algebra due to Lemma \ref{NHFessiminelsonalgebra}.

Now we take a morphism $f:(\mathbf H,F)\longrightarrow(\mathbf H',F')$. Since $f$ is a Heyting homomorphism, we have by straightforward calculations that:
$\alpha(f)((a,b)\to(c,d))=(f(a\rr c),f(a\land d))=(f(a),f(b))\to((f(c),f(d)))=$
$\alpha(f)(a,b)\to\alpha(f)(c,d)$, and similarly for the rest of the operations. 

For all $(a,b)$ $\in$ $N(H,F)$,  $(a,b)=(Id_{(H,F)}(a),Id_{(H,F)}(b)))$, therefore $\alpha(Id_{(\mathbf H,F)})=Id_{\alpha(\mathbf H,F)}$.

Finally, for any two given morphisms $f:(\mathbf H',F')\longrightarrow(\mathbf H'',F'')$, and $g:(\mathbf H,F)\longrightarrow(\mathbf H',F')$, we have that  for all $(a,b)\in N(H,F)$, $\alpha(f\circ g)(a,b)=((f\circ g)(a),(f\circ g)(b))=(f(g(a)),f(g(b)))=\alpha(f)((g(a),g(b)))=((\alpha(f))\circ(\alpha(g)))(a,b)$.
\end{proof}

\begin{Def}
Let $A\in\sN$. We say that an element $a\in A$ is \emph{positive} if $ \sim a\leq a$. We denote by $A^+$ the set of positive elements of $A$.
\end{Def}

\begin{Obs}\label{formadelospositivos}
If $A\in\sN$, then $A^+=\{x\lor\sim x : x\in A\}$. Indeed, if $x\in A^+$, then $\sim x\leq x$, so $x=x\lor\sim x$. On the other hand, if $y=x\lor\sim x$, by $SN$\ref{identidad_ditrib_neg}) and $SN$\ref{identidad_kleene}) it follows that $y=x\lor\sim x\geq x\land\sim x=\sim(\sim x\lor x)=\sim y$.
\end{Obs}

We write $\class{A^+}$ for the set $\{\class{a}:a\in A^+\}$.

\begin{Lemma} \label{lospositivossonifiltros}
If $\mathbf{A}\in\sN$, then $\class{A^+}$ is an $i$-filter of $\mathbf{sH}(A)$.
\end{Lemma}

\begin{proof}
By Theorem \ref{isosemiHeytingsemiNelson}, there exists a subalgebra $\mathbf S$ of $\mathbf{V_k}(\mathbf{sH}(A))$ such that $\mathbf A$ is isomorphic to $\mathbf S$, and the isomorphism is given by $h(a)=(\class{a},\class{\sim a})$ for all $a\in A$. Therefore $S=\{(\class{a},\class{\sim a}) : a\in A\}$. Then, it is immediate that $\pi_1(S)=\mathbf{sH}(A)$.

By Lemma \ref{16012020}, $S=N(\mathbf{sH}(A),E)$, where $E=\{\class{a}\lor\class{\sim a} : a\in A\}$ is an $i$-filter of $\mathbf{sH}(\mathbf A)$. But $E=\class{A^+}$, due to observation \ref{formadelospositivos}, which completes the proof.
\end{proof}

\begin{Prop} \label{functorfromSNtoSHF}
The application $\beta:\mathbf{sN}\longrightarrow\mathbf{sHF}$ defined on the objects of $\mathbf{sN}$ by $\beta(\mathbf{A})=(\mathbf{sH}(\mathbf A),\class{A^+})$, and on morphisms  $f:A\longrightarrow A'$ by $\beta(f)(\class{a})=\class{f(a)}$ for all $a\in A$,  is a  functor from $\mathbf{sN}$ to $\mathbf{sHF}$.
\end{Prop}

\begin{proof}
Let us check that $\beta$ is well defined. For an object $\mathbf{A}$ of $\mathbf{sN}$, $\mathbf{sH}(\mathbf{A})$ is a semi-Heyting algebra, and by Lemma \ref{lospositivossonifiltros} we have that $\class{A^+}$ is an $i$-filter of $\mathbf{sH}(\mathbf{A})$.
Therefore, $(\mathbf{sH}(A),\class{A^+})$ is an object in $\mathbf{sHF}$.

For any semi-Nelson morphism $f:A\longrightarrow A'$, we calculate:
$ \beta(f)(\class{a\to b})=$
$\class{f(a\to b)}=$ $\class{f(a)\to f(b)}=$ $\class{f(a)}\rr\class{f(b)}=$ 
$\beta(f)(\class{a})\rr\beta(f)(\class{b})$,
 and similarly for the other operations.

Also, if $\class{x}\in \class{A^+}$, then by observation \ref{formadelospositivos} we can write $\class{x}=$ $\class{y\lor$ $\sim y}$, for some $y\in A$, and so $\beta(f)(\class{x})=\beta(f)(\class{y\lor\sim y})=\beta(f)(\class{y})\cup\beta(f)(\class{\sim y})=\class{f(y)}\cup\class{f(\sim y)}=$ $\class{f(y)}\cup\class{\sim f(y)}=$ $\class{f(y)\lor\sim f(y)}$. Since $f(y)\in f(A)\subseteq A'$, it follows that $\beta(f)(\class{x})\in \class{A'^+}$. This shows that $\beta(f)(\class{A^+})\subseteq \class{A'^+}$, so $\beta(f)$ is a morphism in $\mathbf{sHF}$.

It is straightforward to check that $\beta$ preserves identities and compositions, so it is indeed a functor.
\end{proof}

The connection between theorems \ref{functorfromSHFtoSN} and \ref{functorfromSNtoSHF} is the following:

\begin{Teo}
The functors $\alpha$ and $\beta$ establish an equivalence between the categories $\mathbf{sHF}$ and $\mathbf{sN}$.
\end{Teo}

\begin{proof}

Notice that $\beta\alpha(\ba,F)=\beta(\mathbf{N(\ba,F)})=(\mathbf{sH(N(\ba,F))},\class{N(A,F)^+})$, with $\mathbf{sH(N(\ba,F))} =\{\class{(x,y)} : (x,y)\in N(A,F)\}=
\{\class{(x,y)} : x,y\in A, x\land y=0, x\lor y\in F \}$.

We define now $\eta_A(x)=\class{(x,x^*)}$, for each object $\mathbf{A}$ of $\mathbf{sHF}$ and for all $x\in A$, and we prove that $\eta=\{\eta_A\}$ is a natural isomorphism from $id_\mathbf{sHF}$ to $\beta\alpha$.

$\eta_A$ is well defined: Since $x\land x^*=0$ and also $\Ds(A)\subseteq F$ it follows that $x\lor x^*\in \Ds(A)\subseteq F$. Hence, $(x,x^*)\in N(A,F)$.

$\eta_A$ is surjective:    
    If $\class{(x,y)}\in\beta\alpha(A,F)$, then $\class{(x,y)}=\class{(x,x^*)}$. Indeed, since $(x,y)\to(x,x^*)=(x\rr x,$ $x\land x^*)=(1,0)$ and $(x,x^*)\to(x,y)=(x\rr x,x\land y)=(1,0)$, so by definition we have that $\class{(x,y)}=\class{(x,x^*)}$. Therefore $\eta_A(x)=\class{(x,y)}$.
    
   $\eta_A$ is injective:
        If $\class{(x,x^*)}=\class{(y,y^*)}$, then $(x,x^*)\to(y,y^*)=(x\rr y,x\land y^*)=(1,0)$. Therefore, from $x\rr y=1$, using properties of semi-Heyting algebras we obtain $x\leq y$. Analogously, $(y,y^*)\to(x,x^*)=$ $(y\rr x,y\land x^*)=(1,0)$ implies $y\leq x$.
        
    $\eta_A$ is a $\mathbf{sHF}$-morphism:
    \begin{itemize}
        \item $\eta_A(x)\cup\eta_A(y)=\class{(x,x^*)}\cup\class{(y,y^*)}=\class{(x,x^*)\sqcup(y,y^*)}=\class{(x\lor y,x^*\land y^*)}=\class{(x\lor y,$ $(x\lor y)^*)}=\eta_A(x\lor y)$.
        \item $\eta_A(x)\cap\eta_A(y)=\class{(x,x^*)}\cap\class{(y,y^*)}=\class{(x,x^*)\sqcap(y,y^*)}=\class{(x\land y,x^*\lor y^*)}$. Also, $\eta_A(x\land y)=\class{(x\land y,(x\land y)^*)}$. Since $(x\land y,x^*\lor y^*)\to(x\land y,(x\land y)^*)=((x\land y)\rr(x\land y),(x\land y)\land(x\land y)^*)=(1,0)$ and $(x\land y,(x\land y)^*)\to(x\land y,x^*\lor y^*)=((x\land y)\rr(x\land y),(x\land y)\land(x^*\lor y^*))=(1,(x\land y\land x^*)\lor(x\land y \land y^*))=(1,0)$, we have that $\eta_A(x)\cap\eta_A(y)=\eta_A(x\land y)$.
        \item $\eta_A(x)\rr \eta_A(y)=\class{(x,x^*)}\rr\class{(y,y^*)}=\class{(x,x^*)\rr(y,y^*)}=\class{(x\rr y,x\land y^*)}$. Also, $\eta_A(x\rr y)=\class{(x\rr y,(x\rr y)^*)}$. Since $(x\rr y,x\land y^*)\to(x\rr y,(x\rr y)^*)=((x\rr y)\rr(x\rr y),(x\rr y)\land(x\rr y)^*)=(1,0)$ and $(x\rr y,(x\rr y)^*)\to(x\rr y,x\land y^*)=((x\rr y)\rr(x\rr y),(x\rr y)\land (x\land y^*))\underset{SH\ref{infimoeimplicaSH})}{=}(1,x\land (y\land y^*))=(1,0)$, we have that $\class{(x\rr y,x\land y^*)}=\class{(x\rr y,(x\rr y)^*)}$. Therefore, we conclude that $\eta_A(x)\rr \eta_A(y)=\eta_A(x\rr y)$.
        \item $\eta_A(F)\subseteq \class{N(A,F)^+}$
            
        If $y\in\eta_A(F)$, then $y=\eta_A(x)$, for some $x\in F$. Hence, $y=\class{(x,x^*)}$, with $x\in F$.
        
        We observe that for $x\in F$, we have that $(x,0)\in N(A,F)$. Besides, $\class{(x,x^*)}=\class{(x,0)}$ because $(x,0)\to(x,x^*)=(x\rr x,x\land x^*)=(1,0)$ and $(x,x^*)\to(x,0)=(x\rr x,x\land 0)=(1,0)$.
            
        Therefore, we can write $y=\class{(x,0)}$. Now, since $\sim(x,0)=(0,x)\leq(x,0)$, it follows that $(x,0)\in N(A,F)^+$, and hence $y\in \class{N(A,F)^+}$.
    \end{itemize}
Finally, we check the naturality of $\eta$, that is, the following diagram commutes:
		\begin{displaymath}
		\xymatrix{ A \ar[rr]^{\eta_A}\ar[d]_{f} && \beta\alpha(A) \ar[d]^{\beta\alpha(f)} \\
			A' \ar[rr]^{\eta_{A'}} && \beta\alpha(A')  }
		\end{displaymath}
    
    Notice that since $\eta_A$ is an isomorphism, $\eta_A^{-1}(\class{(x,y)})=x$ for every $x\in A$. Then we have $$(\eta_{A'}^{-1}\circ \beta\alpha(f)\circ\eta_A)(x)=\eta_{A'}^{-1}(\beta\alpha(f)(\class{(x,0)}))=$$
    $$=\eta_{A'}^{-1}(\class{\alpha(f)(x,0)})=\eta_{A'}^{-1}(\class{f(x),0)}=f(x).$$
    
Thus $f=\eta_{A'}^{-1}\circ \beta\alpha(f)\circ\eta_A$, and therefore $\eta_{A'}\circ f=\beta\alpha(f)\circ\eta_A$.

Now we consider the functor $\alpha\beta$, which transforms every object $\mathbf A$ of $\mathbf{sN}$ into $\mathbf{ N(sH(A),\class{A^+})}$. On a morphism $f:A\longrightarrow A'$ it is defined by  $\alpha\beta(f)(\class{x},\class{y})=(\class{f(x)},\class{f(y)})$. We define for each object $\mathbf A$ of $\mathbf{sN}$, $\delta_A:A\longrightarrow \mathbf{ N(sH(A),\class{A^+})}$  by $\delta_A(x)=(\class{x},\class{\sim x})$.

{$\delta_A$ is well defined:}
        On the one hand, $\class{x}\cap\class{\sim x}=\class{x\land\sim x}=\class{0}$, because $(x\land\sim x)\to 0=(x\land\sim x)\to (x\land\sim x\land0)=(x\land\sim x)\ton 0=1$ and $0\ton (x\land\sim x)=1$ by well known properties of Nelson algebras.
    On the other hand, using $SN$\ref{identidad_kleene}) and De Morgan properties we have $\sim(x\lor\sim x)=\sim x\land x\leq x\lor\sim x$, so $x\lor\sim x\in A^+$. Therefore $\class{x}\cup\class{\sim x}=\class{x\lor\sim x}\in \class{A^+}$.

    {$\delta_A$ is surjective:}
        If $(\class{x},\class{y})\in\alpha\beta(A,F)$, then $(\class{x},\class{y})=(\class{x},\class{\sim x})$. Indeed, this is a consequence of $(\class{x},\class{y})\to(\class{x},\class{\sim x})=(\class{x}\rr \class{x},\class{x}\cap \class{\sim x})=(\class{1},\class{x\land\sim x})=(\class{1},\class{0})$ and $(\class{x},\class{\sim x})\to(\class{x},\class{y})=(\class{x}\rr \class{x},\class{x}\cap \class{y})=(\class{1},\class{0})$. Therefore $\delta_A(x)=(\class{x},\class{y})$.
    
    {$\delta_A$ is injective:}
        If $(\class{x},\class{\sim x})=(\class{y},\class{\sim y})$, then $\class{x}=\class{y}$ and $\class{\sim x}=\class{\sim y}$, so we have the following: $x\to y=1$, $y\to x=1$, $\sim x\to \sim y=1$, $\sim y\to \sim x=1$. By proposition \ref{propertiesseminelson} (\ref{equivalenciadeclasesnelson}), we have that $x\ton y=1$, $y\ton x=1$, $\sim x\ton \sim y=1$, $\sim y\ton \sim x=1$, which implies (see for example \cite{viglizzo99algebras}, (1.15)) that $x=y$.
    
    {$\delta_A$ is a $\mathbf{sN}$-morphism:}
    \begin{itemize}
        \item $\delta_A(x)\sqcup\delta_A(y)=(\class{x},\class{\sim x})\sqcup(\class{y},\class{\sim y})=(\class{x}\cup\class{y},\class{\sim x}\cap\class{\sim y})=(\class{x\lor y},\class{\sim x\land\sim y})=(\class{x\lor y},\class{\sim (x\lor y)})=\delta_A(x\lor y)$. In a similar manner we can prove that $\delta_A(x)\sqcap\delta_A(y)=\delta_A(x\land y)$.
        \item $\sim\delta_A(x)=\sim(\class{x},\class{\sim x})=(\class{\sim x},\class{x})=\delta_A(\sim x)$.
        \item $\delta_A(x)\to \delta_A(y)=(\class{x},\class{\sim x})\to(\class{y},\class{\sim y})=(\class{x}\rr\class{y},\class{x}\cap\class{\sim y})=(\class{x\to y},\class{x\land\sim y})$.
        
        By $SN$\ref{identidad_semi_neg_implica_impl}) and $SN$\ref{identidad_semi_impl_implica_neg}), we see that $\class{x\land\sim y}=\class{\sim(x\to y)}$. Hence $\delta_A(x)\to \delta_A(y)=(\class{x\to y},\class{x\land\sim y})=(\class{x\to y},\class{\sim(x\to y)})=\delta_A(x\to y)$.
    \end{itemize}
    
    Finally, $\delta$ is a natural transformation: to see that this diagram commutes, 
    	\begin{displaymath}
    	\xymatrix{ A \ar[rr]^{\delta_A}\ar[d]_{f} && \alpha\beta(A) \ar[d]^{\alpha\beta(f)} \\
    		A' \ar[rr]^{\delta_{A'}} && \alpha\beta(A')  }
    	\end{displaymath}
    we calculate
        $\alpha\beta(f)\circ\delta_A(x)=\alpha\beta(f)(\class{x},\class{\sim x})=(\class{f(x)},\class{f(\sim x)})=(\class{f(x)},\class{\sim f(x)})=\delta_{A'}\circ f(x)$.
\end{proof}

\begin{Example}\label{EJC1}
Consider the semi-Nelson algebra $\mathbf{A}$, with Hasse diagram indicated below, in which the operations $\to$ and $\sim$ are given by the following tables:
	
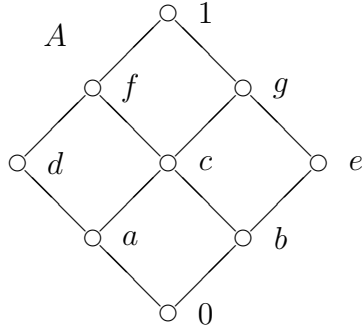
\begin{minipage}{6.4cm}%distancia entre esta minipage y la siguiente
	\begin{center}
			\setlength\unitlength{1mm}
			% \begin{picture}(90,60)(-10,-10)
			\begin{picture}(90,60)(10,-10)
			% puntos
			\put(25,20){\circle{2}}
			\put(35,30){\circle{2}}
			\put(45,40){\circle{2}}
			\put(55,30){\circle{2}}
			\put(45,20){\circle{2}}
			\put(35,10){\circle{2}}
			\put(65,20){\circle{2}}
			\put(55,10){\circle{2}}
			\put(45,0){\circle{2}}
			
			% rotulos a los puntos
			\put(50,0){\makebox(0,0){$0$}}
			\put(40,10){\makebox(0,0){$a$}}
			\put(60,10){\makebox(0,0){$b$}}
			\put(30,20){\makebox(0,0){$d$}}
			\put(50,20){\makebox(0,0){$c$}}
			\put(70,20){\makebox(0,0){$e$}}
			\put(40,30){\makebox(0,0){$f$}}
			\put(60,30){\makebox(0,0){$g$}}
			\put(50,40){\makebox(0,0){$1$}}
			
			% lineas/
			\put(46,21){\line(1,1){8}}
			\put(36,11){\line(1,1){8}}
			\put(26,21){\line(1,1){8}}
			\put(36,31){\line(1,1){8}}
			\put(46,1){\line(1,1){8}}
			\put(56,11){\line(1,1){8}}
			% lineas\
			\put(34,11){\line(-1,1){8}}
			\put(44,21){\line(-1,1){8}}
			\put(54,31){\line(-1,1){8}}
			\put(44,1){\line(-1,1){8}}
			\put(54,11){\line(-1,1){8}}
			\put(64,21){\line(-1,1){8}}

			\put(30,37){\makebox(0,0){$A$}}
			\end{picture}
	\end{center}
\end{minipage}
\begin{minipage}{6cm}
	\begin{center}
			\begin{tabular}{c||c|c|c|c|c|c|c|c|c}
				$\to$ & $0$ & $a$ & $b$ & $c$ & $d$ & $e$ & $f$ & $g$ & $1$ \\ \hline\hline
				 $0$  & $1$ & $1$ & $1$ & $1$ & $g$ & $1$ & $g$ & $1$ & $g$ \\ \hline
				 $a$  & $1$ & $1$ & $1$ & $1$ & $g$ & $1$ & $g$ & $1$ & $g$ \\ \hline
				 $b$  & $1$ & $1$ & $1$ & $1$ & $g$ & $1$ & $g$ & $1$ & $g$ \\ \hline
				 $c$  & $1$ & $1$ & $1$ & $1$ & $g$ & $1$ & $g$ & $1$ & $g$ \\ \hline
				 $d$  & $e$ & $g$ & $e$ & $g$ & $1$ & $e$ & $1$ & $g$ & $1$ \\ \hline
				 $e$  & $d$ & $d$ & $f$ & $f$ & $a$ & $1$ & $c$ & $1$ & $g$ \\ \hline
				 $f$  & $e$ & $g$ & $e$ & $g$ & $1$ & $e$ & $1$ & $g$ & $1$ \\ \hline
				 $g$  & $d$ & $d$ & $f$ & $f$ & $a$ & $1$ & $c$ & $1$ & $g$ \\ \hline
				 $1$  & $0$ & $a$ & $b$ & $c$ & $d$ & $e$ & $f$ & $g$ & $1$
			\end{tabular}
	\end{center} 
\end{minipage}

	\begin{center}
		\begin{tabular}{c|ccccccccc}
			$x$&$0$&$a$&$b$&$c$&$d$&$e$&$f$&$g$&$1$\\\hline
			$\sim x$&$1$&$g$&$f$&$c$&$e$&$d$&$b$&$a$&$0$\\
		\end{tabular}
	\end{center} 

$S=\{0,a,d,e,g,1\}$ is the universe of a subalgebra $\mathbf{S}$ of $\ba$. Both $\mathbf{S}$ and $\ba$ have the same quotient semi-Nelson algebra $\mathbf{H}$ indicated below:

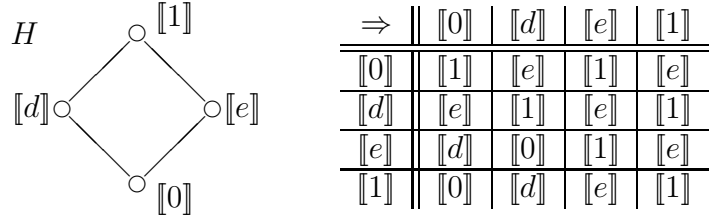
\begin{minipage}{6.5cm}%distancia entre esta minipage y la siguiente
	\begin{center}
		\setlength\unitlength{1mm}
		\begin{picture}(90,60)(10,-10)
		% puntos
		\put(55,30){\circle{2}}
		\put(45,20){\circle{2}}
		\put(65,20){\circle{2}}
		\put(55,10){\circle{2}}
		
		% rotulos a los puntos
		\put(60,8){\makebox(0,0){$\class{0}$}}
		\put(41,20){\makebox(0,0){$\class{d}$}}
		\put(69,20){\makebox(0,0){$\class{e}$}}
		\put(60,32){\makebox(0,0){$\class{1}$}}
		
		% lineas/
		\put(46,21){\line(1,1){8}}
		\put(56,11){\line(1,1){8}}
		% lineas\
		\put(54,11){\line(-1,1){8}}
		\put(64,21){\line(-1,1){8}}
		
		\put(40,30){\makebox(0,0){$H$}}
		\end{picture}
	\end{center}
\end{minipage}
\begin{minipage}{6cm}
	\begin{center}
		\begin{tabular}{c||c|c|c|c}
			   $\rr$    & $\class{0}$ & $\class{d}$ & $\class{e}$ & $\class{1}$ \\ \hline\hline
			$\class{0}$ & $\class{1}$ & $\class{e}$ & $\class{1}$ & $\class{e}$ \\ \hline
			$\class{d}$ & $\class{e}$ & $\class{1}$ & $\class{e}$ & $\class{1}$ \\ \hline
			$\class{e}$ & $\class{d}$ & $\class{0}$ & $\class{1}$ & $\class{e}$ \\ \hline
			$\class{1}$ & $\class{0}$ & $\class{d}$ & $\class{e}$ & $\class{1}$ \\ 
		\end{tabular}
	\end{center} 
\end{minipage}

$E=\{\class{e},\class{1}\}$ is an i-filter of $\mathbf{H}$, and $\mathbf{N(H,E)}$ is isomorphic to $\mathbf{S}$, while $\mathbf{N(H,H)}=\mathbf{V_k(H)}$ is isomorphic to $\mathbf{A}$.
\end{Example}

\section{Dually Hemimorphic semi-Nelson Algebras}\label{duallyhemimorphicsection}

In this section we will show some results about the variety of dually hemimorphic semi-Heyting and semi-Nelson algebras, which were presented in definitions \ref{defdualsemiheyting} and \ref{defdualseminelson} respectively. We omit some proofs of the following results, since it can be founded in \cite{cornejo19dually}.

\begin{Lemma}
Let $\mathbf{A}\in\dsn$. If we consider the equivalence relation $\equiv$ defined in section \ref{quotients}, then that relation is compatible with respect of the operation $'$. Hence, $\equiv$ is a congruence in $\dsn$.
\end{Lemma}

In view of the previous lemma, we can define $$\class{x}\tic=\class{x'}$$ and moreover, it can be proved that $\mathbf{dsH}(A)=\langle A/\equiv;\cap,\cup,\rr,\tic,\bot,\top \rangle\in\dsh$.

\begin{Def}\label{def_op_vak_tic}
Let $\mathbf{A}=\langle A;\cap,\cup,\rr,\tic,\bot,\top \rangle\in\dsh$. For $(a,b)\in V_k(A)$ we define the following unary operation on $V_k(A)$: $$(a,b)'=(a\tic,b\cap(a\tic\rr a)).$$
\end{Def}

The operation is well defined. Indeed, using $SH\ref{infimoeimplicaSH}$) we have $a\tic\cap b\cap(a\tic\rr a)=b\cap a\tic\cap(a\tic\rr a)$ $=b\cap a\tic\cap a=0$.

\begin{Teo}
If $\mathbf{A}\in\dsh$, then the system $\mathbf{V_k(A)}=\langle V_k(A); \sqcap,\sqcup,\to,\sim,',\top\rangle$ is a dually hemimorphic semi-Nelson algebra.
\end{Teo}

In a similar manner as on section \ref{02022020}, we have the following representations:

\begin{Teo} \label{isoduallysemiHeytingsemiNelson} {\rm \cite{cornejo19dually} }
If $\mathbf{A}\in\dsh$, then $\mathbf{A}$ is isomorphic to $\mathbf{dsH}(\mathbf{V_k}(\mathbf A))$. Also, if $\mathbf{A}\in\dsn$, then $\mathbf{A}$ is isomorphic to a subalgebra of $\mathbf{V_k(dsH(A))}$.
\end{Teo}

Now we will generalize the results of section \ref{representations}. In \cite{cornejo18categorical} the authors introduced the category $\sN_c$ of centered semi-Nelson algebras whose objects are algebras $\mathbf{A}=\langle A; \land,\lor,\to,\sim,c,1 \rangle$ of type $(2,2,2,1,0,0)$ such that $c=\sim c$, and the morphisms are the algebra homomorphisms. The element $c$ is called a \emph{center} of the algebra $A$, and it is necessarily unique. We can expand the category $\sN_c$ to $\dsn_c$ of dually hemimorphic centered semi-Nelson algebras whose objects are algebras $\mathbf{A}=\langle A; \land,\lor,\to,\sim,',c,1 \rangle$ of type $(2,2,2,1,1,0,0)$ such that $c=\sim c$, and the morphisms are the algebra homomorphisms.

We will prove that the categories $\dsn$ and $\dsn_c$ actually coincide.

\begin{Prop}\label{duallyheminelsonarecentered}
If $\mathbf A\in\dsn$, then $\mathbf A$ is centered, with center $1'$.
\end{Prop}

\begin{proof}
We want to prove that $\sim 1'=1'$, or equivalently, the following conditions holds simultaneously:
\[
    \sim 1'\to1'=1,\
    1'\to\sim1'=1,
\]
\begin{equation}\label{19922020_5}
    \sim1'\to\sim(\sim1')=1,\
    \sim(\sim1')\to\sim1'=1.
\end{equation}

By Proposition \ref{propertiesseminelson} (\ref{equivalenciadeclasesnelson}), (\ref{19922020_5}) is equivalent to
\[
    \sim 1'\ton1'=1,\
    1'\ton\sim1'=1,\
\]
\[
    \sim1'\ton\sim(\sim1')=1,\
    \sim(\sim1')\ton\sim1'=1,
\]
which is equivalent to show that
\begin{equation}\label{19922020_6}
    \sim 1'\ton1'=1,
\end{equation}
\begin{equation}\label{19922020_7}
    1'\ton\sim1'=1.
\end{equation}

By \textit{(DSN\ref{dhsn_4})}, \textit{(DSN\ref{dhsn_5})}, and proposition \ref{equivalenciadeclasesnelson}, we have that $\sim 1'\ton(\sim 1\land(1'\to 1))=1$, that is (1) $\sim 1'\ton0=1$. Since in every Nelson algebra $0\ton x=x$ holds for all $x$, it follows that (2) $0\ton 1'=1'$. By (1), (2) and proposition \ref{propertiesseminelson} (\ref{transitivitynelson}), we have that $\sim 1'\ton1'=1$, which proves \eqref{19922020_6}.

By \textit{(DSN\ref{dhsn_2})} we have (3) $1=1'\to \sim1=1'\to0=1'\to(1'\land0)=1'\ton0$. Using again that $0\ton 1'=1'$, and proposition \ref{propertiesseminelson} (\ref{transitivitynelson}), gives us $\sim 1'\ton1=1$, which proves \eqref{19922020_7}.
\end{proof}

\vspace{3mm}

Proposition \ref{duallyheminelsonarecentered} shows us that the category $\dsn$ it can be viewed as the category $\dsn_c$ with a distinguished element $c$, but despite the change of the language, both categories are essentially the same.

It was proved in \cite{cornejo18categorical} that there is a categorial equivalence between semi-Heyting algebras and centered semi-Nelson algebras. Moreover, the function $h$ defined in the proof of Theorem \ref{isoduallysemiHeytingsemiNelson} is an isomorphism. This gives us the following:

\begin{Teo} \label{isoduallysemiNelson}
If $\mathbf{A}\in\dsn$, then $\mathbf{A}$ is isomorphic to $\mathbf{V_k(dsH(A))}$.
\end{Teo}
\vspace{1mm}
\begin{Obs}
Another way of proving Theorem \ref{isoduallysemiNelson} is the following: if we consider the structure $\mathbf{N(H,E)}$ as in Lemma \ref{16012020}, and take $\mathbf{H}\in\dsh$, we have no guarantee that this set is going to be a closed under the operation $^\dagger$. If instead we consider a subalgebra $\mathbf S$ such that $\pi_1(S)=H$, then the proof indicated in the lemma still holds for the variety $\dsh$, since it does not use the operations $'$ or $\tic$. Hence, the sets $S$ and $N(H,E)$ are equal, so $N(H,E)$ inherits the operations from $\mathbf S$. Indeed, due to Theorem \ref{isoduallysemiHeytingsemiNelson}, we know that the function $h:A\to \mathbf{V_k}(\mathbf{dsH}(A))$ defined by $h(a)=(\class{a},\class{\sim a})$ is a monomorphism. Hence, $A\cong h(A)$. Let us take then the subalgebra $\mathbf S=h(\mathbf A)$ of $\mathbf{V_k}(\mathbf{dsH}(\mathbf A))$. It is immediate that $\pi_1(S)=\mathbf{dsH}(A)$. Therefore, by Lemma \ref{16012020}, $h(\mathbf  A)=\mathbf{N(dsH(A),E)}$ for $E=\{\class{x}\in \mathbf{dsH}(A):\class{x}=\class{a}\lor \class{b},(\class{a},\class{b})\in S\}$.

Since $\mathbf S$ is a subalgebra of $\mathbf{V_k}(\mathbf{dsH}(\mathbf A))$, we have that $(\class{1},\class{0})\in S$, so the center $(\class{0},\class{0})=(\class{1},\class{0})'$ is an element of $S$. Then $\class{0}=\class{0}\lor\class{0}\in E$. Since $E$ is a filter, this shows that $E=\mathbf{dsH}(A)$.

We have proved that $S=N(\mathbf{dsH}(A),\mathbf{dsH}(A))=\mathbf{V_k}(\mathbf{dsH}(A))$, and therefore $h$ is an isomorphism.
\end{Obs}

\begin{Example}\label{EJC2}
Consider the semi-Nelson algebras $\mathbf{A}$ and its subalgebra  $\mathbf{S}$ from Example \ref{EJC1}. We know by the results above that $\mathbf{S}$ cannot be expanded by adding a unary operation $'$ in such a way that the system $\mathbf{B}=\langle B; \land,\lor,\to,\sim,',1 \rangle$ becomes a dually hemimorphic semi-Nelson algebra, since $\mathbf{S}$ has no center.

On the other hand, any operation $'$ that makes $\mathbf{sH(A)}$ into a dually hemimorphic semi-Heyting algebra yields an operation $^\dagger$ that will turn $A$ into a dually hemimorphic semi-Nelson algebra.
\end{Example}

\bibliographystyle{alpha}

\end{document}